\documentclass[a4paper,11pt]{article}
\usepackage[T1]{fontenc}
\usepackage[utf8x]{inputenc} 
\usepackage{lmodern}
\usepackage{amsmath}
\usepackage{amsthm}
\usepackage{amssymb}
\usepackage{authblk}
\usepackage{graphicx}
\usepackage{enumitem}
\usepackage{xcolor}

\usepackage{amsthm}
\usepackage{amssymb}
\usepackage{xparse}
\usepackage{thmtools}
\usepackage{stackrel}
\usepackage{bbold}
\usepackage{relsize}
\usepackage{tikz}
\usepackage{bbm}

\usetikzlibrary{hobby}

\newcommand{\R}{\mathbb{R}}
\newcommand{\C}{\mathbb{C}}

\newcommand{\N}{\mathbb{N}}

\newcommand{\dd}{\mathrm{d}}

\renewcommand{\S}{\mathbb{S}}

\DeclareMathOperator{\diag}{diag}

\DeclareMathOperator{\tr}{tr}

\DeclareMathOperator{\dist}{dist}

\makeatletter
\newtheoremstyle{indented}
{7pt} 
{7pt} 
{} 
{1.5em} 
{\bfseries} 
{.} 
{.5em} 
{} 

\theoremstyle{definition}

\newtheorem{defn}{Definition}[section]

\theoremstyle{plain}
\newtheorem*{theorem*}{Theorem}

\newtheorem{theorem}{Theorem}

\newtheorem{prop}[defn]{Proposition}

\newtheorem{lem}[defn]{Lemma}

\theoremstyle{definition}
\newtheorem{rem}[defn]{Remark} 

\makeatletter
\renewcommand*\env@matrix[1][*\c@MaxMatrixCols c]{%
  \hskip -\arraycolsep
  \let\@ifnextchar\new@ifnextchar
  \array{#1}}
\makeatother

\usepackage{a4wide}
\title{The Bergman kernel in constant curvature}
\author{Alix Deleporte\thanks{deleporte@math.unistra.fr}}
\affil{Universit\'e de Strasbourg, CNRS, IRMA UMR 7501, F-67000 Strasbourg, France}

\usepackage{bbm}

\usepackage{scrtime}

\newcommand{\fixitemthm}{\leavevmode\vspace{-0.5em}} 

\newcommand\blfootnote[1]{%
  \begingroup
  \renewcommand\thefootnote{}\footnote{#1}%
  \addtocounter{footnote}{-1}%
  \endgroup
}

\begin{document}

\maketitle

\blfootnote{This work was supported by grant
  ANR-13-BS01-0007-01}
\begin{abstract}
We present an elementary proof for an approximate
expression of the Bergman kernel on homogeneous spaces, and products
of them. The error term is exponentially small with respect to the
inverse semiclassical parameter. 
\end{abstract}

\section{Introduction}
\label{sec:introduction}

\subsection{Bergman kernels}
\label{sec:toeplitz-operators}
This article is devoted to the study of the Bergman kernel on
homogeneous spaces, that is, K\"ahler manifolds with constant
curvature (see Definition \ref{def:homo}). This class of manifolds
contain complex projective spaces (on which the Bergman kernel is
explicit), as well as tori and hyperbolic manifolds (on which it is
not). This kernel encodes the
holomorphic sections of a
suitable line bundle over $M$.

The study of the Bergman kernel is mainly
motivated by \emph{Berezin-Toeplitz quantization}, which associates to
a function $f$ on $M$ a sequence $(T_N(f))_{N\in \N}$ of linear
operators on holomorphic sections over $M$. Toeplitz operators allow to tackle
problems arising from representation theory \cite{charles_knot_2015}, semiclassical analysis
\cite{sjostrand_density_1996} and quantum spin systems \cite{deleporte_low-energy_2017}. The Bergman
kernel is also associated with determinantal processes \cite{berman_determinantal_2008}, 
sampling theory
\cite{berman_fekete_2011},
and nodal sets
\cite{zelditch_large_2013}.

\begin{defn}\label{def:toep}\fixitemthm
  \begin{itemize}
    \item A K\"ahler manifold $(M,J,\omega)$ is
      \emph{quantizable} when there exists a Hermitian line bundle $(L,h)$ over $M$ with curvature $-2i\pi
\omega$. The bundle $(L,h)$ is then called prequantum line bundle over $M$.
\item Let $(M,J,\omega)$ be a quantizable K\"ahler manifold
  with $(L,h)$ a prequantum bundle and let
  $N\in \N$.
  \begin{itemize}
    \item The Hardy space $H_0(M,L^{\otimes N})$ is the
space of holomorphic sections of $L^{\otimes N}$. It is a closed subspace of
$L^2(M,L^{\otimes N})$ which consists of all square-integrable
sections of the same line bundle.

  \item The Bergman projector $S_N$ is the orthogonal projector from
  $L^2(M,L^{\otimes N})$ to $H_0(M,L^{\otimes N})$.
\end{itemize}
\end{itemize}
\end{defn}
The simplest example of a quantizable compact K\"ahler manifold is the
one-dimensional projective space $\C\mathbb{P}^1$, endowed with the
natural complex structure $J_{st}$ and the Fubini-Study form
$\omega_{FS}$. A natural bundle over $\C\mathbb{P}^1$ is the
tautological bundle (the fibre over one point is the corresponding
complex line in $\C^2$). Then $L$ is the \emph{dual} of the tautological
bundle. One can show that $H_0(\C\mathbb{P}^1,L^{\otimes N})$ is
isomorphic to the space of
homogeneous polynomials of degree $N$ in two variables (with scalar
product given by the volume form on $\S^3$).

The space $H_0(M,L^{\otimes N})$ is always finite-dimensional if $M$
is compact. Indeed,
since $\Delta=-\partial\overline{\partial}$, one has
\[
  H_0(M,L^{\otimes N})=\ker
  \overline{\partial}_{C^{\infty}(M,L^{\otimes N})\to
    \Omega_1(M,L^{\otimes N})}\subset \ker \Delta_{C^{\infty}(M,L^{\otimes N})\to
    \Omega_1(M,L^{\otimes N})}.
\]
The Laplace operator $\Delta$ is elliptic on the compact manifold $M$, so that its kernel is finite-dimensional.

\subsection{Kernels of linear operators between sections of line
  bundles}
\label{sec:kern-line-oper}
The Bergman projector $S_N$ is a linear operator mapping the space of
sections $H_0(M,L^{\otimes N})$, to itself. Here we
describe what it means for such an operator to have an integral
kernel, and the nature of this kernel.

If $E$ and $F$ are finite-dimensional vector spaces, then it is well
known that the space $L(F,E)$ of linear opeators from $E$ to $F$ can
be identified with $F\otimes E^*$ where $E^*$ is the dual of
$E$. Using this, let us construct, for any two line bundles
$E_1\stackrel{\pi_1}{\to}M_1$ and $E_2\stackrel{\pi_2}{\to}M_2$ over
Riemannian manifolds, a
space of kernels $E_1\boxtimes E_2^*$ for linear operators which associate, to a section of
$E_2$, a section of $E_1$.

The space $E_1\boxtimes E_2^*$ will be constructed as a vector bundle
over $M_1\times M_2$. An informal definition is that the fiber
$(E_1\boxtimes E_2^*)_{(x,y)}$ over a point $(x,y)\in M_1\times M_2$ is
defined as the tensor product $(E_1)_x\otimes (E_2)_y^*$.

One can formally build $E_1\boxtimes E_2^*$ in two steps. The first
step is to associate to $E_1\stackrel{\pi_1}{\to}M_1$ a bundle $E_1'\stackrel{\pi_1'}{\to}M_1\times M_2$ as
follows:
\begin{align*}
  E_1'&=E_1\times M_2\\
  \pi_1'(e,y)&=(\pi_1(e),y).
\end{align*}
Then $(E_1')_{(x,y)}=(\pi_1')^{-1}((x,y))=\pi_1^{-1}(x)\times \{y\}\simeq (E_1)_x$. Similarly, from the dual bundle
$E_2^*$ of $E_2$, one can build
$E_2^{*'}\stackrel{\pi_2'}{\to}M_1\times M_2$. Then, the second step
is to define
\[
  E_1\boxtimes E_2^*=E_1'\otimes E_2^{*'}.
\]
Then the fibre over one point reads
\[
  (E_1\boxtimes E_2^*)_{(x,y)}\simeq(E_1')_{(x,y)}\otimes
  (E_2^{*'})_{(x,y)}\simeq(E_1)_x\otimes (E_2)_y^*,
\]
as prescribed.

A smooth section of $E_1\boxtimes E_2^*$ gives a linear operator between
compactly supported, smooth sections of $E_2$ and sections of
$E_1$. Indeed, if $K_A$ is a smooth section
of $E_1\boxtimes E_2^*$, then for any compactly supported, smooth section $s$ of $E_2$ one can
define the section $As$ of $E_1$ as
\[
  (As)(x)=\int_{M_2}K_A(x,y)s(y)\dd Vol(y).
\]
Indeed, $K_A(x,y)\in(E_1)_x\otimes (E_2)_y^*$ is a linear operator
from $(E_2)_y$ (to which $s(y)$ belongs) and $(E_1)_x$. Then the
integral makes sense as taking values in $(E_1)_x$, so that $As$ is
well-defined as a section of $E_1$. 

In particular, in our setting the Bergman projector $S_N$ admits a kernel as an element of
$L^{\otimes N}\boxtimes \overline{L}^{\otimes N}$. Indeed, since $H_0(M,L^{\otimes N})$ is finite-dimensional, it
is spanned by a Hilbert base $s_1,\ldots,s_{d_N}$ of holomorphic
sections of $L^{\otimes N}$. Then the kernel of $S_N$ is
\[
  S_N(x,y)=\sum_{i=1}^{d_N}s_i(x)\otimes \overline{s_i(y)}.
\]

\subsection{Statement of the main results}
\label{sec:stat-main-results}

  \begin{defn}
\label{def:homo}
  A K\"ahler manifold $(M,\omega, J)$ is called \emph{homogeneous} under the two following conditions:
  \begin{itemize}
  \item For every two points $x,y\in M$, there exist an open set $U\in
    M$ containing $x$, an open set $V\in M$ containing $y$, and a
    biholomorphism $\rho:U\mapsto V$ which preserves $\omega$.
  \item For every point $x\in M$, there exist an open set $U\in M$
    containing $x$ and an action of $U(d)$ by $\omega$-preserving
    biholomorphisms on $U$, with $x$ as only common fixed point, such that the
    induced linear action on $T_xM$ is conjugated to the tautological
    action of $U(d)$ on $\C^d$.
  \end{itemize}
\end{defn}
There is a one-parameter family of local models for homogeneous
manifolds of dimension $d$: for positive curvature $c>0$, the rescaled
complex projective space $\C\mathbb{P}^d$; for zero curvature $c=0$,
the vector space $\C^d$; for negative curvature $c>0$, the rescaled
hyperbolic space $\mathbb{H}^{2d}$. In particular, on a homogeneous
K\"ahler manifold $(M,\omega,J)$, in the real-analytic structure given by $(M,J)$,
the symplectic form $\omega$ is real-analytic. 

Using the standard notion of holomorphic extensions of
real-analytic functions on totally real submanifolds, let us define
what will be the kernel of the Bergman projector, up to a constant
multiplicative factor and an exponentially small error.

\begin{defn}[A particular section of $L^{\otimes N}\boxtimes
  \overline{L}^{\otimes N}$]
  \label{def:natural-section}
The bundle $L\boxtimes \overline{L}$, when restricted to the diagonal
$M_{\Delta}=\{(x,y)\in M\times M,x=y\}$,
is the trivial line bundle $M\times \C\to M$. Moreover, if the first
component of $M\times M$ is endowed with the complex structure on $M$,
and the second component with the opposite complex structure (we
informally call $M\times \overline{M}$ this complex manifold), then
$M_{\Delta}$ is a totally real submanifold of $M\times \overline{M}$.

Over a small neighbourhood of $M_{\Delta}$ in $M\times M$, one can
then uniquely define $\Psi^1$ as the unique holomorphic section of
$L\boxtimes \overline{L}$ which is equal to $1$ on $M_{\Delta}$.

This section is locally described at follows: let $s$ be a
non-vanishing holomorphic section of $L$ over a small open set
$U\subset M$. Let $\phi=-\frac 12 \log|s|_h$. Then $\phi$ is
real-analytic, so that it admits a holomorphic extension
$\widetilde{\phi}$, defined on $U\times \overline{U}$ (again, the
diagonal copy of $U$ is totally real in $U\times \overline{U}$). Then
\[
  \Psi^1(x,y)=e^{2\widetilde{\phi}(x,y)}s(x)\otimes \overline{s(y)}.
\]

We then define $\Psi^N$ as $(\Psi^1)^{\otimes N}$, which is a section of $L^{\otimes N}\boxtimes
\overline{L}^{\otimes N}$.
  \end{defn}

\begin{theorem}
\label{thr:Szeg-homo}
  Let $M$ be a quantizable K\"ahler manifold of complex dimension $d$ and
  suppose $M$ is a product of compact homogeneous K\"ahler manifolds.
  
  Then the Bergman projector $S_N$ on $M$
  has an approximate kernel: there is a sequence of real coefficients
  $(a_i)_{0\leq i \leq d}$, and positive constants $c,C$ such
  that, for all $(x,y)\in M\times M$ and for all $N\geq 1$, one has
    \[
    \left\|S_N(x,y)-\Psi^N(x,y)\sum_{k=0}^{d}N^{d-k}a_k\right\|_{h}\leq Ce^{-cN}.
  \]
  If $M$ is homogeneous, with curvature $\kappa$, then
  \[
    \sum_{k=0}^{d}N^{d-k}a_k=\frac{1}{\pi^d}(N-\kappa)(N-2\kappa)\ldots(N-d\kappa).
    \]
  \end{theorem}
 A proof of Theorem \ref{thr:Szeg-homo} using advanced microlocal
 analysis (local Bergman kernels) was first hinted in
\cite{berman_sharp_2012} and detailed in
\cite{hezari_off-diagonal_2017}, where the coefficients $a_k$ are
explicitely computed through an explicit expression of the K\"ahler
potential $\phi$ in a chart. We propose to prove Theorem
\ref{thr:Szeg-homo} without semiclassical tools, and to recover the
coefficients $a_k$ from an elementary observation of the case of
positive curvature.
  
Theorem \ref{thr:Szeg-homo} implies exponential approximation in the
$L^2$ operator sense. Indeed, if $K$ is a section of $L^{\otimes N}\boxtimes
\overline{L}^{\otimes N}$ with $\|K(x,y)\|_h\leq C$ for all $(x,y)\in
M^2$, then for $u\in L^2(M,L^{\otimes N})$ one has
\begin{align*}
  \int_M\left\|\int_M\langle K(x,y),u(y)\rangle_h\dd y\right\|_h^2\dd
  x&\leq \int_M\int_M\|\langle K(x,y),u(y)\rangle_h\|_h^2\dd y \dd x\\
   &\leq \int_M\int_M \|K(x,y)\|_h^2\|u(y)\|_h^2\dd x \dd y\\
  &\leq C^2Vol(M)\|u\|^2_{L^2}.
\end{align*}

Expressions for the Bergman kernel such as the one appearing in
Theorem \ref{thr:Szeg-homo} were first obtained by Charles
\cite{charles_berezin-toeplitz_2003} in the smooth setting; in this weaker
case the section $\Psi^N$ is only defined at every order on the
diagonal, which yields an $O(N^{-\infty})$ remainder.

Our proof of Theorem \ref{thr:Szeg-homo},
does not rely on microlocal analysis; the only partial
differential operator involved is the Cauchy-Riemann operator
$\overline{\partial}$ acting on $L^2(M,L^{\otimes N})$. We use the
following estimate on this operator: if $M$ is compact, there exists $C>0$ such that, for every
$N\geq 1$ and $u\in L^2(M,L^{\otimes N})$, one has:
\begin{equation}
  \label{eq:Kohn}
  \|\overline{\partial}u\|_{L^2}\geq C\|u-S_Nu\|_{L^2}.
\end{equation}
This estimate follows from the work of Kohn \cite{kohn_harmonic_1963,kohn_harmonic_1964}, which relies only
on the basic theory of unbounded operators on Hilbert spaces; it is
widely used in the asymptotic study of the Bergman kernel, where it is
sometimes named after H\"ormander or Kodaira.



The rest of this article is devoted to the proof of Theorem
\ref{thr:Szeg-homo}. The plan is the following: we build an approximation
$\widetilde{S}_N$, up to exponential precision, for the Bergman
kernel on compact homogeneous spaces. The method consists in constructing
candidates $\widetilde{\psi}_{x,v}^N$ for the coherent states, using
the local symmetries. These states are almost
holomorphic and satisfy a reproducing condition; from these
properties, we deduce that the associated reproducing kernel is
exponentially close to the Bergman kernel.

\begin{rem}[Non-compact homogeneous spaces]
  Since Kohn's estimate \eqref{eq:Kohn} is valid for general
  homogeneous manifolds, the method of approximation of the Bergman
  kernel which we provide in this paper adapts to non-compact
  homogeneous spaces under the condition that the radius of
  injectivity is bounded from below. In the simple picture of
  hyperbolic surfaces of finite genus, we allow for the presence of funnels but not
  cusps (more specifically, the behaviour of the Bergman kernel far
  away in a cusp, where the diameter is smaller than $N^{-\frac
    12}$, is unknown to us). The exact statement of Theorem
  \ref{thr:Szeg-homo} is valid in this context, however we cannot
  conclude that $S_N$ is controlled in the $L^2$ operator norm.
\end{rem}

\section{Radial holomorphic charts}

K\"ahler potentials on a K\"ahler manifold $(M,J,\omega)$ are characterised by the
following property. If $\rho$ is a local
holomorphic chart for $M$, the pulled-back symplectic form $\rho^*\omega$ can
be seen as a function of $\C^d$ into anti-Hermitian matrices of size
$2d$. The closure condition $d\omega=0$ is then equivalent to the existence of a
real-valued function $\phi$ on the chart such that $i\partial
\overline{\partial}\phi=\rho^*\omega$. Such a $\phi$ is a
K\"ahler potential.

From now on, $(M,J,\omega)$
denotes a compact quantizatble homogeneous K\"ahler manifold, of
complex dimension $d$, and
$(L,h)$ is the prequantum bundle over $M$.

Near every
point $P_0\in M$, we will build a \emph{radial holomorphic chart}
using the local homogeneity. This chart is the main ingredient in the
construction of the approximate coherent states.

\begin{prop}
  \label{prop:rad-hol-chart}
  For every $P_0\in M$, there is an open set $U\subset M$ with
  $P_0\in U$, an open set $V\subset \C^d$ invariant under $U(d)$, and a biholomorphism
  $\rho:V\mapsto U$, such that $\rho^* \omega$ is invariant under
  $U(n)$.

  In particular, in this chart, there exists a K\"ahler potential
  $\phi$ which depends only on the distance to the origin, with
  real-analytic regularity.
\end{prop}
\begin{proof}
  Let $\rho_0:V_0\mapsto U_0$ be any local holomorphic chart to a
  neighbourhood of $P_0$, with $\rho_0(0)=P_0$.
  
  Since $M$ is homogeneous, there exists an open set $P_0\in
  U_1\subset U_0$ and an action of $U(n)$ on $U_1$  such that, for any $g\in U(d)$, one has
  \begin{align*}
    D(x\mapsto \rho_0^{-1}(g\cdot \rho_0(x)))(0)&=g\\
    (g\cdot)^*J&=J\\
    (g\cdot)^*\omega &= \omega.
  \end{align*}
  In particular, for $g\in U(d)$, the map $\rho_g:x \mapsto g\cdot
  \rho_0(g^{-1}x)$ is a biholomorphism from $V_2=\bigcap_{g\in U(d)}g\circ
  \rho_0^{-1}(U_1)$ onto its image $U_2(g)$.

  For $x\in \bigcap_{g\in U(d)}U_2(g)$, let us define
  \[
    \sigma(x)=\int_{U(d)}\rho_g^{-1}(x) \dd\mu_{Haar}(g).
  \]
  Then $D(\sigma\circ \rho_0)(0)=I$. Hence, $\sigma$ is a biholomorphism,
  from a small $U(d)$ invariant open set $U\ni P_0$ into a small
  $U(d)$ invariant open set $V\ni 0$. By construction $\sigma$ is
  $g$-equivariant, in the sense that $\sigma(gx)=g\cdot
  \sigma(x)$. Then $\sigma^{-1}$ is the requested chart since $\omega$
  is invariant under the action of $U(d)$ on $U$.

  Let us proceed to the second part of the Proposition. We first let
  $\phi_1$ be any real-analytic K\"ahler potential in the chart
  $\sigma^{-1}$. We then define
  \[
    \phi(x)=\int_{g\in U(n)}\phi_1(gx)\dd \mu_{Haar}(g).
  \]
  Then $\phi$ is a radial function since $U(d)$ acts transitively on
  the unit sphere. Moreover, since $\sigma_*\omega$ is $U(d)$-invariant
  then $x\mapsto \phi_1(gx)$ is a K\"ahler potential, so that the mean
  value $\phi$ is a K\"ahler potential.
\end{proof}

\begin{rem}
  There is exactly one degree of freedom in the choice of the chart
  $\rho$ in Proposition \ref{prop:rad-hol-chart}: the precomposition by a scaling
  $z\mapsto \lambda z$ preserves all requested properties. In general, the
  metric $\sigma_*\Re(\omega)$, at zero, is a constant times the
  standard metric. This constant can be modified by the scaling
  above. Hence, without loss of generality, one can choose the chart
  so that the K\"ahler potential has the following Taylor expansion at
  zero:
  \[
    \phi(x)=\frac{|x|^2}{2}+O(|x|^3),
  \]
  so that the metric $\sigma^*g$, at zero, is the standard metric.
\end{rem}

\begin{defn}
  A chart satisfying the conditions of Proposition
  \ref{prop:rad-hol-chart}, such that the radial K\"ahler potential
  has the following Taylor expansion at zero:
  \[
    \phi(x)=\frac{|x|^2}2+O(|x|^3),
  \]
  is called a \emph{radial holomorphic chart}.
\end{defn}

The following elementary fact will be used extensively:
\begin{prop}
  \label{prop:Decay-Kah-pot}
  The radial K\"ahler potential $\phi$ of a radial holomorphic chart is strongly convex. In particular, for
  all $x\neq 0$ in the domain of $\phi$ one has $\phi(x)>0$.
\end{prop}
\begin{proof}
  From the Taylor expansion $\phi(x)=\frac{|x|^2}2+O(|x|^3)$, one
  deduces that the real Hessian matrix of $\phi$ is positive definite
  at zero. Near any point $x\neq 0$ which belongs to the domain of
  $\phi$, in spherical coordinates the function $\phi$ depends only on
  the distance $r$ to the origin. The Levi form $\frac{\partial^2
    \phi}{\partial z_i\partial \overline{z_j}}(x)$, which is Hermitian
  positive definite (since $\phi$ is strongly pseudo-convex), is
then equal to $\frac{\partial^2\phi}{\partial r^2}(r)Id$. In
particular, $\frac{\partial^2\phi}{\partial r^2}>0$ everywhere, so that
  $\phi$ is strongly convex at $x$.
\end{proof}
\section{Approximate coherent states}
\label{sec:appr}
We first recall the notion of coherent states in Berezin-Toeplitz
quantization.

 \begin{defn}\label{def:Coh-state}
   Let $(P_0,v)\in
   L$. We define the associated coherent state, which is a section of
   $L^{\otimes N}$, as follows:
   \[
     \psi_{P_0,v}^N=\left(u\mapsto \langle u(P_0),v\rangle_{h}\right)^{*_{H_0(M,L^{\otimes N})}}.
           \]
         \end{defn}
         That is, the evaluation map $u\mapsto \langle u(P_0),v\rangle_{h}$ is a linear operator on $H_0(M,L^{\otimes
                 N})$, and by the Riesz representation theorem, there
               exists $\psi_{P_0,v}^N$ such that linear map is
               $\langle \psi_{P_0,v}^N,\cdot\rangle$.

Let us use the radial charts above to build an approximation for coherent states on
a homogeneous K\"ahler manifold.

\begin{prop}\label{prop:radius}
  There exists
  $r>0$ such that the following is true.
  \vspace{-0.5em}
  \begin{enumerate}
\item 
  Let $P_0\in M$. There exists a radial holomorphic chart near $P_0$,
  whose domain contains $B(0,r)$.

\item Let $\phi$ denote the radial K\"ahler potential near $P_0$. For all $N\geq 1$ the quantity \[a(N)=\int_{B(0,r)}\exp(-N\phi(|z|))\dd z
    \dd \overline{z}\] is well-defined and does not depend on $P_0$.
  \end{enumerate}
\end{prop}
\begin{proof}
  \fixitemthm
  \begin{enumerate}
  \item Let $P_1\in M$. By Proposition
    \ref{prop:rad-hol-chart} there exists a radial holomorphic chart
    near $P_1$. 
    Since $M$ is homogeneous, a small neighbourhood of any
    $P_0\in M$, of size independent of $P_0$ since $M$ is compact, can be mapped into a neighbourhood of 
    $P_1\in M$. By restriction of the radial holomorphic chart of
    Proposition \ref{prop:rad-hol-chart} to this neighbourhood, whose preimage contains a
    small ball around zero, this defines a chart around $P_0$. Since
    $M$ is compact, there is a radius $r$ such that, for every $P_0\in
    M$, the closed ball $\overline{B(P_0,r)}$ is
    contained in the domain of the chart around $P_0$.
  \item By construction of the chart above, the K\"ahler potential
    $\phi$ does not depend on $P_0$. Moreover, $\phi$ is a smooth
    function on $\overline{B(0,r)}$, hence the claim..
  \end{enumerate}
\end{proof}
\begin{rem}
  We will see at the end of the proof of Theorem \ref{thr:Szeg-homo}
  that $a(N)^{-1}$ is
  exponentially close to a polynomial in $N$.
\end{rem}

From now on, $r$ is as in the claim of Proposition \ref{prop:radius}.

\begin{prop}
   Let $(P_0,v)\in L$. The action of $U(n)$ on a
  neighbourhood $U$ of $P_0$ in $M$ can be lifted in an action on
  $L_U$.
\end{prop}
\begin{proof}
  By definition of $L$, if $V$ is the preimage of $U$ by a radial
  holomorphic chart, the bundle $(L_U,h)$ is isomorphic to
  \[
    (V\times \C,\exp(-\phi(z))|u|^2).
  \]
  Since $\phi$ is invariant under $U(n)$, the linear action of $U(n)$
  on $V$ can be trivially extended to $V\times \C$ and preserves the
  metric.
\end{proof}

In order to treat local holomorphic sections of a prequantum bundle
over a quantizable compact homogeneous K\"ahler manifold, let
us define the Ancillary space and the approximate coherent states:

\begin{defn}
  Let $\phi$ be the radial K\"ahler potential on $M$ and $r$ be as
  in Proposition \ref{prop:radius}. Let $N\in \N$. The ancillary space is defined as
  \[
    A_N=\left\{u\text{ holomorphic on
      $B(0,r)$},\int_{B(0,r)}e^{-N\phi(z)}|u|^2\leq +\infty\right\}.
\]
It is a Hilbert space with the scalar product
\[
  \langle
  u,v\rangle_{A_N}=\int_{B(0,r)}e^{-N\phi(z)}u(z)\overline{v(z)}\dd z.
\]

  The set $A_N$ consists of functions belonging to the usual Hardy space of the
  unit ball, but the scalar product is twisted by the K\"ahler
  potential $\phi$.

  Since the function $\phi$ appearing in the definition of $A_N$ is a
  universal local K\"ahler potential on $M$, for each $(P_0,v)\in L^*$
  there is a natural isomorphism (up to
  multiplication of all norms by $\|v\|_{h}$)
  $\mathfrak{S}_{P_0,v}^N$ between $A_N$
  and the space of $L^2$ local holomorphic sections $H_0(U,L^{\otimes
    N})$ where $U=\sigma_{P_0}^{-1}(B(0,r))$. We define
  $\widetilde{\psi}_{P_0,v}^N$ as the element of
  $H_0(U,L^{\otimes N})$ associated with the constant function
  $a(N)^{-1}\in A_N$.

We set $\widetilde{\psi}_{P_0,v}^N$ to be zero outside
$\sigma^{-1}(B(0,r))$ so that $\widetilde{\psi}_{P_0,v}^N\in
L^2(M,L^{\otimes N})$.
The function $\widetilde{\psi}_{P_0,v}^N$ is equivariant with respect
to $v$: one has
\[
  \widetilde{\psi}_{P_0,v}^N=\left(\overline{v}/\overline{v'}\right)^N\widetilde{\psi}_{P_0,v'}^N.\]
This allows us to define the approximate normalized coherent state
$\widetilde{\psi}_{P_0}$ as an element of $L^2(M,L^{\otimes N})\otimes
\overline{L}_{P_0}^{\otimes N}$.
\end{defn}

Let us prove that the approximate coherent states are very close to
$H_N(M,L)$:
\begin{prop}
  \label{prop:appr-close-HN}
  There exists $c>0$ and $C>0$ such that, for all $P_0\in M$,
  \[
    \|S_N\widetilde{\psi}_{P_0}^N-\widetilde{\psi}_{P_0}\|_{L^2}\leq
    Ce^{-cN}.
  \]
\end{prop}
\begin{proof}
  Let $\chi$ denote a test function on $\R$ which is smooth and such
  that $\chi=1$ on $[0,\frac r2]$ and $\chi=0$ on $[r,+\infty)$.

  The section $(\chi\circ|\sigma|) \widetilde{\psi}_{P_0}^N$ is smooth; since $\widetilde{\psi}_{P_0}^N$ is holomorphic on $\sigma^{-1}(B(0,r))$ and decays
    exponentially fast far from $P_0$, one has
    \[
      \|\overline{\partial}(\chi\circ|\sigma|) \widetilde{\psi}_{P_0}^N\|_{L^2}\leq Ce^{-cN}.
    \]

    From Kohn's estimate \eqref{eq:Kohn} we deduce that
     \[
    \|S_N(\chi\circ|\sigma|)\widetilde{\psi}_{P_0}^N-(\chi\circ|\sigma|)\widetilde{\psi}_{P_0}\|_{L^2}\leq
    Ce^{-cN}.
  \]
  In addition, since $\phi>c$ on $B(0,r)\setminus B(0,r/2)$, one has
  \[
    \|(\chi\circ|\sigma|) \widetilde{\psi}_{P_0}^N-\widetilde{\psi}_{P_0}^N\|_{L^2}\leq Ce^{-cN}.
  \]

  Since $S_N$ is an orthogonal projector, its operator norm is bounded
  by $1$, so that the previous estimates implies
  \[
    \|S_N(\chi\circ |\sigma|) \widetilde{\psi}_{P_0}^N-S_N\widetilde{\psi}_{P_0}^N\|_{L^2}\leq Ce^{-cN}.
  \]
  This ends the proof.
\end{proof}
To show that our approximate coherent states are indeed exponentially
close to the actual coherent states we will use the following lemma.

\begin{lem}
  \label{lem:lin-form-AN}
  Any continuous linear form on $A_N$ invariant by linear unitary
  changes of variables is
  proportional to the continuous linear form $v\mapsto \langle
  v,1\rangle$.

  In particular, the continuous linear form $A_N\ni u\mapsto u(0)$ is
  equal to the scalar product with the constant function $a(N)^{-1}$
\end{lem}
\begin{proof}
  A Hilbert basis of $A_N$ is given by the normalised monomials
  $e_{\nu}z\mapsto c_{\nu}z^{\nu}$ for $\nu\in \N^d$, for some $c_{\nu}>0$. Special elements of $U(n)$ are the
  diagonal matrices $\diag(\exp(i\theta_1),\ldots,\exp(i\theta_d))$
  which send $e_{\nu}$ into $\exp(i\theta\cdot \nu)e_{\nu}$.

  A linear form $\eta$ invariant under $U(d)$ must be such
  that $\eta(e_{\nu})=\exp(i\theta\cdot \nu)\eta(e_{\nu})$ for every
  $\theta,\nu$. In particular, $\nu\neq 0 \Rightarrow
  \eta(e_{\nu})=0$.
  Since $\eta$ is continuous we deduce that $\eta$ is proportional to
  the scalar product with $e_0=c_01$.

  For the second part of the Proposition we only need to prove that
  the multiplicative factor between the two continuous
  $U(d)$-invariant linear forms of
  $A_N$, evaluation at $0$ on one side, scalar product with
  $a(N)^{-1}$ on the other side, is $1$. By Definition of
  $a(N)$, the scalar product in $A_N$ of $a(N)^{-1}$ with $a(N)^{-1}$ is
  $a(N)^{-1}$, moreover the evaluation at zero of $a(N)^{-1}$ is
  $a(N)^{-1}$, hence the claim.
\end{proof}


The functions $\widetilde{\psi}_{P_0,v}^N$
mimic the definition of coherent states.

\begin{prop}\label{prop:reproducing}
  There exists $c>0$ such that, for any $(P_0,v_0),(P_1,v_1)\in L^*$,
  \begin{itemize}
    \item If $\dist(P_0,P_1)\leq \frac r2$, then
      \[
        \left|\langle
          \widetilde{\psi}_{P_1,v_1}^N,\widetilde{\psi}_{P_0,v_0}^N\rangle-\langle
          \widetilde{\psi}_{P_1,v_1}(P_0),v_0^{\otimes
          N}\rangle_{h}\right|=O(e^{-cN}).
      \]
        \item In general, one has
    \[|\langle
      \widetilde{\psi}_{P_1,v_1}^N,\widetilde{\psi}_{P_0,v_0}^N\rangle|\leq
      Ce^{-cN\dist(P_0,P_1)^2}.
      \]
  \end{itemize}
\end{prop}
\begin{proof}\fixitemthm
  \begin{itemize}
    \item
  The continuous linear functional on $A_N$ which sends $u$ to $u(0)$
  is invariant under the action of $U(n)$ (since $0$ is a fixed
  point), so that, by Lemma \ref{lem:lin-form-AN}, it is proportional to the scalar
  product with a constant. This property, read in the map
  $\mathfrak{S}_{P_0,v_0}^N$, means that, for every $(P_1,v_1)\in L$ the
  scalar product \[\langle
    \widetilde{\psi}_{P_0,v_0}^N,\widetilde{\psi}_{P_1,v_1}^N\rangle\] is
  a constant (independent of $P_1$) times
  $\langle S_N\widetilde{\psi}_{P_1,v_1}(P_0),v_0^{\otimes
    N}\rangle_{h}$. The normalizing
  factor $a(N)$ is such that both sides are equal to $1$ if
  $P_1=P_0$. This ends the proof since $S_N$ is almost identity on the
  almost coherent states.
  
\item If $\dist(P_0,P_1)\geq 2r$ then $\widetilde{\psi}_{P_0,v_0}^N$ and
  $\widetilde{\psi}_{P_1,v_1}^N$ have disjoint support so that the scalar product
  is zero.

  If $r/2\leq \dist(P_0,P_1)\leq 2r$ then $\widetilde{\psi}_{P_1,v_1}^N$ is
  exponentially small on $B(P_0,r/4)$ and $\widetilde{\psi}_{P_0,v_1}^N$
  is exponentially small outside this ball so that the scalar product
  is smaller than $Ce^{-cN(4r)^2}$ for some $c>0$.

  If $P_1\in B(P_0,r/2)$, one can apply the previous point; the claim
  follows from the fact that
  $\phi(|x|)\geq c|x|^2$ on $B(P_0,r/2)$.
  \end{itemize}
\end{proof}

\section{Approximate Bergman projector}
\label{sec:appr-szegho-proj}

We can now define the approximate Bergman projector by its kernel:
$\widetilde{S}_N$ is a function on $\overline{L}^{\otimes N}\boxtimes
L^{\otimes N}$ which is linear in the fibres (or, equivalently, a section of $L^{\otimes N}\boxtimes
\overline{L}^{\otimes N}$)
defined by the formula:
\[
\widetilde{S}_N((x,v),(y,v'))=\langle \widetilde{\psi}_{x,v}^N,\widetilde{\psi}_{y,v'}^N\rangle.
\]

We wish to prove that this operator is very close to the actual
Bergman projector, defined by the actual coherent states $\psi_{P_0,v}^N$:
\begin{prop}\label{prop:SNappr=exact}
Let $(P_0,v)\in L$. Then 
  $S_N\widetilde{\psi}_{P_0,v}^N=\psi_{P_0,v}^N.$
  \end{prop}
  \begin{proof}
    Let $U=B(P_0,r)$. By construction, the scalar product of $\widetilde{\psi}_{P_0,v}^N$ with any element of $H_N(U,L^{\otimes N})$ is the value
    at $P_0$ of this element, taken in scalar product with $v$. As $H_N(M,L^{\otimes N})\subset
    H_N(U,L^{\otimes N})$ in a way which preserves the scalar product
    with $\widetilde{\psi}_{P_0,v}^N$, from Definition \ref{def:Coh-state} one has $S_N\widetilde{\psi}_{P_0,v}^N=\psi_{P_0,v}^N.$
  \end{proof}

  From Propositions \ref{prop:appr-close-HN} and  \ref{prop:SNappr=exact} we deduce that approximate
  coherent states are, indeed, close to coherent states. In
  particular,

  \begin{prop}Uniformly on $(x,y)\in M\times M$, there holds
    \[\|\widetilde{S}_N(x,y)-S_N(x,y)\|_{h}=O(e^{-cN}).
      \]
  \end{prop}
  \begin{proof}
    The exact Bergman kernel is expressed in terms of the coherent
    states as:
    \[
      S_N((x,v),(y,v'))=\langle \psi_{x,v}^N,\psi_{y,v'}^N\rangle.
    \]
    From this and the Definition of $\widetilde{S}_N$,
    since \[S_N\widetilde{\psi}_{x,v}^N=\psi_{x,v}^N=\widetilde{\psi}^N_{x,v}+O(e^{-cN}),\]
    the kernels of $S_N$ and $\widetilde{S}_N$ are exponentially close.
  \end{proof}

  \section{Approximate projector in a normal chart}
\label{sec:appr-proj-norm}

To conclude the proof of Theorem \ref{thr:Szeg-homo} in the
homogeneous case, it only remains to compute an approximate expression
for $\widetilde{S}_N(x,y)=\langle
\widetilde{\psi}_x^N,\widetilde{\psi}_y^N\rangle$. At first sight,
this looks easy. Indeed, on the diagonal,
$\widetilde{S}_N(x,x)=a(N)^{-1}$. Moreover $\widetilde{S}_N$ is
$O(e^{-cN})$-close from the Bergman kernel $S_N$, which is holomorphic in the first variable
and anti-holomorphic in the second variable. However, one cannot
conclude that $\widetilde{S}_N$ is exponentially close to the
holomorphic extension of $a(N)^{-1}$ (that is,
$a(N)^{-1}\Psi^N$). Indeed, $S_N(x,x)-a(N)^{-1}$, while
exponentially small, might oscillate very fast, so that its
holomorphic extension is not uniformly controlled.

By studying change of charts between radial holomorphic charts, one can prove the following Proposition.
\begin{prop}\label{prop:expr-brg-ker}
There exists $c>0$ and $C>0$ such that, for all $(x,y)\in M\times M$, there holds
  \[
    \left\|\widetilde{S}_N(x,y)-\Psi^N(x,y)a(N)^{-1}\right\|_{h}\leq Ce^{-cN}.
  \]
\end{prop}
\begin{proof}
It is sufficient to prove the claim for $x,y$ close enough from each other.
  
We first need to understand how to change from the radial holomorphic
chart around $x$ to the radial holomorphic chart around $y$.
By hypothesis, if $x$ and $y$ are two points in $M$ at distance
less than $\frac r2$, if $\rho$ denotes a radial chart at $x$,  there
is a map $\sigma:B(0,\frac r2)\to B(0,r)$, which is biholomorphic on
its image and which preserves the metric $\rho^*g$, and such that
$\sigma(0)=\rho(y)$. The associated holomorphic map on $B(0,\frac
r2)\times \C$
which preserves the Hermitian metric pulled back by $\rho$ on the fibre is of the form:
\begin{equation}
  \label{eq:change-charts-L}
  (z,v)\mapsto
  \left(\sigma(z),\exp\left(\frac
  12(\phi(|z|^2)-\phi(|\sigma(z)|^2))+if_{\sigma}(z)\right)v\right),
  \end{equation}
where $f_{\sigma}$ is such that the function
\[
  m\mapsto \phi(|z|^2)-\phi(|\sigma(z)|^2)+if_{\sigma}(z)
\]
is holomorphic.
Such a $f_{\sigma}$ exists and is unique up to an additive constant:
indeed, since $\sigma$ preserves the metric $g$, $z\mapsto
\phi(|\sigma(z)|^2)$ is a K\"ahler potential on $B(0,\frac
r2)$. Hence, the
map
\[
  z\mapsto \phi(|z|^2)-\phi(|\sigma(z)|^2)
\]
is harmonic, so that it is the real part of a holomorphic function.

Then, by \eqref{eq:change-charts-L}, in a radial holomorphic chart around $x$, the almost coherent state
$\widetilde{\psi}_{y,v'}^N$ is written as
\[
  z\mapsto a(N)^{-1}\mathbb{1}_{V(y)}\overline{v'}\exp\left(-\frac N2\phi(|\sigma(z)|^2)+if_{\sigma}(z)\right).
\]

By Proposition \ref{prop:reproducing}, the scalar product with
$\widetilde{\psi}_{x,v}^N$, with $y$ close to $x$, is
\[
  \langle
  \widetilde{\psi}_{y,v'}^N,\widetilde{\psi}_{x,v}^N\rangle=a(N)^{-1}(v\overline{v'})\exp\left(-\frac
  N2 \phi(|\rho(y)|^2)+iNf_{\sigma}(0)\right)+O(e^{-cN}).
\]

In particular, in a radial holomorphic chart $\rho$ around $x$, the
approximate Bergman kernel evaluated at $x$ has the following form for $z$ small:
\[
  \widetilde{S}_N(\rho(z),\rho(0))=a(N)^{-1}\exp(Ng(z))\psi_{x}^N(\rho(z))\overline{\psi_{x}^N(\rho(0))}+O(e^{-cN}),
\]
where $g$ is holomorphic.
Using another change of charts given by
\eqref{eq:change-charts-L}, the form of the approximate Bergman
kernel, near the diagonal,
is
\[
  \widetilde{S}_N(\rho(z),\rho(w))=a(N)^{-1}\exp(NF(z,w))\widetilde{\psi}_{x}^N(\rho(z))\overline{\widetilde{\psi}_{x}^N(\rho(w))}+O(e^{-cN}),
\]
where $F$ is holomorphic in the first variable and anti-holomorphic in
the second variable.

Moreover, $\widetilde{S}_N(z,z)=\widetilde{S}_N(0,0)=a(N)^{-1}$, hence $F(z,\overline{w})=\widetilde{\phi}(z\cdot
\overline{w})$.

The
expression of the phase in coordinates coincides with the section $\Psi^N$ of
Definition \ref{def:natural-section} (the non-vanishing section $s$
here is $\widetilde{\psi}_{x}^1$). Thus the Bergman kernel can be written as
\[
  \widetilde{S}_N(x,y)=\Psi^N(x,y)a(N)^{-1}+O(e^{-cN}).
\]
\end{proof}

We will compute explicitely $a(N)^{-1}$ in Section
\ref{sec:coeff-bergm-kern}. Up to this computation, the proof of
Theorem $A$ is complete in the case of a single homogeneous manifold.

It remains to prove how to pass from homogeneous manifolds to direct
products of such. This relies on the following Proposition.
\begin{prop}
  Let $M_1,M_2$ be compact quantizable K\"ahler manifolds and
  $L_1,L_2$ be the associated prequantum line bundles. Then
  $L_1\boxtimes L_2$ is the prequantum line bundle over $M_1\times
  M_2$, and
  \[
    H_0(M_1\times M_2,(L_1\boxtimes L_2)^{\otimes N})\simeq
    H_0(M_1,L_1^{\otimes N})\otimes H_0(M_2,L_2^{\otimes N}).
  \]
\end{prop}
\begin{proof}
  There is a tautological, isometric injection
  \[
    \iota:H_0(M_1,L_1^{\otimes N})\otimes H_0(M_2,L_2^{\otimes N})\hookrightarrow H_0(M_1\times M_2,(L_1\boxtimes L_2)^{\otimes N})
  \]
  which is such that, for $(s_1,s_2)\in H_0(M_1,L_1^{\otimes N})\times
  H_0(M_2,L_2^{\otimes N})$ and $(x,y)\in M_1\times M_2$, one has
  \[
    \iota(s_1\otimes s_2)(x,y)=s_1(x)\otimes s_2(y).
  \]
  It remains to prove that any element of $H_0(M_1\times
  M_2,(L_1\boxtimes L_2)^{\otimes N})$ belongs to the image of the
  element above. To this end, let us prove that, for any
  $(x_1,v_1),(x_2,v_2)\in L_1\times L_2$, the coherent state at
  $((x_1,x_2),v_1\otimes v_2)$ is given by
  \[
    \psi^N_{(x_1,x_2),v_1\otimes v_2}=\iota(\psi^N_{x_1,v_1}\otimes
    \psi^N_{x_2,v_2}).\]
  Indeed, for any $s\in H_0(M_1\times
  M_2,(L_1\boxtimes L_2)^{\otimes N})$, one has
  \begin{align*}
    \langle s,\iota(\psi^N_{x_1,v_1}\otimes
    \psi^N_{x_2,v_2})\rangle&=\int_{M_1}\left\langle \int_{M_2}\langle
                              s(y_1,y_2),\psi^N_{x_2,v_2}(y_2)\rangle_{(L_2)_{y_2}^{\otimes N}}
                              \dd
                              y_2,\psi^N_{x_1,v_1}(y_1)\right\rangle_{(L_1)_{y_1}^{\otimes N}}
                              \dd x_1\\
    &=\int_{M_1}\langle s(y_1,x_2),\psi_{x_1,v_1}^{\otimes N}\otimes
      v_2\rangle_{(L_1)_{y_1}^{\otimes N}\otimes (L_2)_{x_2}^{\otimes N}}
      \dd x_1\\
    &=\langle s(x_1,x_2),v_1\otimes v_2\rangle_{(L_1)_{x_1}^{\otimes N}\otimes (L_2)_{x_2}^{\otimes N}}=\langle
      s,\psi_{(x_1,x_2),v_1\otimes v_2}^N\rangle.
  \end{align*}
  The image of $\iota$ thus contains all coherent states on
  $M_1\times M_2$. Hence, the orthogonal of the range of $\iota$ in $H_0(M_1\times
  M_2,(L_1\boxtimes L_2)^{\otimes N})$ is zero, which concludes the proof.
\end{proof}
In particular, the Bergman kernel on a product $M_1\times M_2$ is
given by
\[
  S_N^{M_1\times M_2}(x_1,x_2,y_1,y_2)=S_N^{M_1}(x_1,y_1)\otimes
  S_N^{M_2}(x_2,y_2).
\]
This, along with Propositions \ref{prop:SNappr=exact} and
\ref{prop:expr-brg-ker}, concludes the proof of Theorem
\ref{thr:Szeg-homo} up to the study of $a(N)^{-1}$, which we perfrorm
in the next section.

\section{The coefficients of the Bergman kernel}
\label{sec:coeff-bergm-kern}
Since, for all $x\in M$, one has $\Psi^N(x,x)=1$, then the trace of
the Bergman kernel is given by
\[
  \tr(S_N)=\sum_{i=1}^{d_N}1=\int_M\sum_{i=1}^{d_N}s_i(x)\overline{s_i(x)}\dd
                              x
                            =\int_M S_N(x,x) \dd z
                            = a(N)^{-1}Vol(M)+O(e^{-cN}).
\]
In particular, $a(N)^{-1}$ is exponentially close to an integer
divided by $Vol(M)$. Let
\[
  P(N)=\frac{\tr(S_N)}{Vol(M)}.
  \]

In this section we compute $P(N)$ in the case of a homogeneous
manifold of dimension $d$. Since
\[
  P(N)^{-1}=\int_{B(0,r)}\exp(-N\phi(|z|))\dd z \dd \overline{z}+O(e^{-cN}),
\]
and there is a universal local model for $M$ which depends only on its
curvature $\kappa$, then $P(N)$ depends only on $\kappa$ and the dimension $d$. Moreover, $P(N)^{-1}$ has real-analytic
dependence on $\kappa$. We will give an expression for $P(N)$ which is
valid on $\kappa\in\{\frac{1}{k},\,k\in \N\}.$ Since $P(N)$ is
real-analytic in $\kappa$, it will follow that this expression is valid for
all curvatures. From now on we write $P_{\kappa}(N)$ to indicate that
$P(N)$ depends on $N$ and $\kappa$, and only on them.

Let us consider the case of the rescaled projective space: \[(M_k,\omega_k,J)=(\C\mathbb{P}^d,k\omega_{FS},J_{st}).\]
This space is quantizable; the prequantum bundle is
simply
\[
  L_k=(L_1)^{\otimes k},
\]
so that
\[
  S_{N,k}(x,y)=S_{Nk,1}(x,y).
\]
Moreover, the curvature of $(M_k,\omega_k)$ is $\frac{1}{k}$. In other
terms,
\[
  P_{\frac 1k}(N)=\frac{Vol(M_1)}{Vol(M_k)}P_1(kN)=k^{-d}P_1(kN).
\]
It remains to compute $P_1$. On $\C\mathbb{P}^d$, the prequantum
bundle $L_1$ is explicit: it is $O(1)$, the dual of the tautological
line bundle. In this setting,
\[
  H_0(M,L^{\otimes N})\simeq \C_N[X_1,\ldots,X_d].
\]
Hence,
\[
  P_1(N)=\frac{1}{Vol(\C\mathbb{P}^d)}\dim(\C_N[X_1,\ldots,X_d])=\frac{d!}{\pi^d}\binom{N+d}{d}=\frac{1}{\pi^d}(N+1)\ldots(N+d).
  \]
  Hence, for any $\kappa$ of the form $\frac 1k$ with $k\in \N$ there holds
  \[
    P_{\kappa}(N)=\frac{1}{\pi^d}(N+\kappa)(N+2\kappa)\ldots(N+d\kappa).
  \]
  Since $P_{\kappa}$ has real-analytic dependence on $\kappa$, the
  formula above is true for any $\kappa\in \R$, which concludes the
  proof.

  \section{Acknowledgements}
\label{sec:acknowledgements}

The author thanks N. Anantharaman and L. Charles for useful discussion.







\bibliographystyle{abbrv}
\bibliography{main}
\end{document}